\documentclass[12pt]{amsart}
\usepackage{amssymb,amsmath,amsfonts}              
\usepackage{mathrsfs}
\usepackage{graphicx}
\usepackage{hyperref}
\usepackage{amsthm}
\usepackage{multicol}
\usepackage{color}
\usepackage{comment}
\usepackage{float,wrapfig}
\usepackage{lipsum}
\usepackage{epstopdf}
\usepackage{enumerate}
\renewcommand{\epsilon}{\varepsilon}
\DeclareMathOperator{\dvg}{div} 
\DeclareMathOperator{\dist}{dist} 
\DeclareMathOperator{\graph}{graph} 

\DeclareMathOperator{\spa}{span}

\DeclareMathOperator{\G}{G}

\def\pd{\partial}
\def\cd{\nabla}
\def\R{\mathbb{R}}

\def\N{\mathbb{N}}

\def\d{\delta}

\def\M{\mathcal{M}}

\newcommand{\inner}[2]{\left\langle#1,#2\right\rangle} 

\def\ba #1\ea {\begin{align} #1\end{align}}
\def\bann #1\eann {\begin{align*} #1\end{align*}}
\def\ben #1\een {\begin{enumerate} #1\end{enumerate}}
\def\bi #1\ei {\begin{itemize}\renewcommand\labelitemi{--} #1\end{itemize}}

\newtheorem{theorem}{Theorem}[section]

\newtheorem{proposition}[theorem]{Proposition}

\newtheorem{remark}[theorem]{Remark}

\newtheorem{corollary}[theorem]{Corollary}

\newtheorem{claim}[theorem]{Claim}

\newtheoremstyle{TheoremNum}
        {\topsep}{\topsep}              
        {\itshape}                      
        {}                              
        {\bfseries}                     
        {.}                             
        { }                             
        {\thmname{#1}\thmnote{ \bfseries #3}}
    \theoremstyle{TheoremNum}

\title[Concavity of the arrival time]{Differential Harnack inequalities via Concavity of the arrival time}

\author{Theodora Bourni}
\author{Mat Langford}

\address{Department of Mathematics, University of Tennessee Knoxville, Knoxville TN, 37996-1320}
\email{tbourni@utk.edu}

\address{Department of Mathematics, University of Tennessee Knoxville, Knoxville TN, 37996-1320}
\address{Department of Mathematics, The University of Newcastle, Newcastle, NSW, Australia, 2308}
\email{mlangford@utk.edu}
\email{mathew.langford@newcastle.edu.au}

\begin{document}

\begin{abstract}
We present a simple connection between differential Harnack inequalities for hypersurface flows and natural concavity properties of their time-of-arrival functions. We prove these concavity properties directly for a large class of flows by applying a concavity maximum principle argument to the corresponding level set flow equations. In particular, this yields a short proof of Hamilton's differential Harnack inequality for mean curvature flow and, more generally, Andrews' differential Harnack inequalities for certain ``$\alpha$-inverse-concave'' flows.
\end{abstract}

\maketitle
\thispagestyle{empty}





\vspace{-6pt}

\section{Concavity maximum principles}

Our goal is to deduce concavity properties for the time-of-arrival functions of a large class of geometric flow equations using the concavity maximum principle. The main idea, due to Korevaar \cite{Korevaar83} and later extended by Kennington \cite{Kennington85} and Kawohl \cite{Kawohl85} is summarized in the following theorem. 

\begin{theorem}\label{thm:concavity MP}
Let $\smash{\Omega\subset \R^n}$ be a bounded, convex, open set and suppose that $\smash{u\in C^1(\overline\Omega)}$ is twice differentiable in $\smash{\Omega}$ and satisfies the equation
\[
-f(Du(x),D^2u(x))=b(x,u(x),Du(x))\;\;\text{in}\;\;\Omega
\]
with $\smash{f:\R^n\times \Gamma\to\R}$, $\Gamma\underset{\mathrm{convex, open}}{\subset}\mathrm{Sym}^{n\times n}$, 
satisfying
\begin{itemize}
\item[(i)] Weak ellipticity:
\[
r\ge s \;\; \implies \;\; f(p,r)\ge f(p,s)\,. 
\]
\item[(ii)] Concavity: 
\[
f(p,\lambda r+(1-\lambda)s)\ge \lambda f(p,r)+(1-\lambda)f(p,s)\,. 
\]
\end{itemize}
and $\smash{b:\Omega\times \R\times \R^n\to\R}$ satisfying
\begin{itemize}
\item[(iii)] Monotonicity:
\[
z>w \;\ \implies \;\; b(x,z,p)<b(x,w,p)\,.
\]
\item[(iv)] Joint concavity:
\[
b(\lambda(x,z)+(1-\lambda)(y,w),p)\geq \lambda b(x,z,p)+(1-\lambda)b(y,w,p)\,.
\]
\end{itemize}
If the graph of $u$ lies below its boundary tangent hyperplanes, then $u$ is concave.
\end{theorem}
\begin{proof}
The argument is essentially that of Korevaar \cite{Korevaar83}: Consider Korevaar's ``concavity function'' $Z:[0,1]\times\Omega\times\Omega\to\R$, defined by \cite{Korevaar83}
\begin{equation}\label{eq:concavity function}
Z(r,x,y):= u(rx+(1-r)y)-\big(ru(x)+(1-r)u(y)\big)\,.
\end{equation}
This function measures how far the point $\big(rx+(1-r)y,u(rx+(1-r)y)\big)$ in $\overline\Omega\times\R$ lies above the line joining the points $(x,u(x))$ and $(y,u(y))$. We need to prove that $Z\geq 0$. 

Choose the triple $(r_0,x_0,y_0)$ so that
\[
Z(r_0,x_0,y_0)=\min_{[0,1]\times\overline\Omega\times\overline \Omega}Z(r,x,y)\,.
\]
If $r_0=0$ or $r_0=1$, then $Z(r_0,x_0,y_0)=0$, which implies the claim. So we may assume that $r_0\in(0,1)$. Suppose that $x_0\in \partial \Omega$. If $Z(r_0,x_0,y_0)<0$, then, since the graph of $u$ lies below its boundary tangent hyperplanes, it would be possible to find a point $(r_1,x_1,y_1)$ with $Z(r_1,x_1,y_1)<Z(r_0,x_0,y_0)$ by moving $x_0$ a small amount inwards along the line joining $x_0$ and $y_0$, contradicting minimality of $(r_0,x_0,y_0)$ \cite{Korevaar83}. Indeed, consider the function
\[
f(\varepsilon):=Z(x_\varepsilon,y_0,r_\varepsilon)\,,
\]
where $x_\varepsilon:=x_0+\varepsilon(y_0-x_0)$ and $r_\varepsilon:=r_0/(1-\varepsilon)$. Since
\[
r_\varepsilon x_\varepsilon+(1-r_\varepsilon)y_0\equiv z_0\,,
\]
the boundary condition implies that
\bann
\left.\frac{d}{d\varepsilon}\right|_{\varepsilon=0}f={}&r_0\big(u(y_0)-u(x_0)-Du(x_0)\cdot(y_0-x_0)\big)\\
<{}&0\,,
\eann
in contradiction with the fact that $f(\varepsilon)$ is minimized at $\varepsilon=0$.

A similar argument applies at $y_0$. So we may assume that $x_0$ and $y_0$ are interior points. 

Let us abuse notation by writing $Z(x,y):=Z(r_0,x,y)$. Then $(x_0,y_0)$ is a stationary point of $Z$ and hence, setting $z_0:= r_0x_0+(1-r_0)y_0$,
\begin{subequations}
\begin{equation}\label{eq:gradZ x}
0=\partial_{x^i}Z(x_0,y_0)=r_0(u_i(z_0)-u_i(x_0))
\end{equation}
and
\begin{equation}\label{eq:gradZ y}
0=\partial_{y^i}Z(x_0,y_0)=(1-r_0)(u_i(z_0)-u_i(y_0))\,.
\end{equation}
\end{subequations}
So
\begin{equation}\label{eq:classical concavity MP gradient condition}
Du(z_0)=Du(x_0)=Du(y_0)=:p_0\,.
\end{equation}
Since $(x_0,y_0)$ is a local minimum,
\begin{align*}
0\leq{}&(\partial_{x^i}+\partial_{y^i})(\partial_{x^j}+\partial_{y^j})Z(x_0,y_0)\\
={}&u_{ij}(z_0)-r_0u_{ij}(x_0)-(1-r_0)u_{ij}(y_0)\,.
\end{align*}
The ellipticity and concavity of $f$ and the joint-concavity of $b$ then imply
\begin{align*}
b(z_0,u(z_0),p_0)={}&-f(p_0,D^2u(z_0))\\
\leq{}&-f(p_0,r_0D^2u(x_0)+(1-r_0)D^2u(y_0))\\
\leq{}&-r_0f(p_0,D^2u(x_0))-(1-r_0)f(p_0,D^2u(y_0))\\
={}&r_0b(x_0,u(x_0),p_0)+(1-r_0)b(y_0,u(y_0),p_0)\\
\leq {}&b(z_0,r_0u(x_0)+(1-r_0)u(y_0),p_0)\,.
\end{align*}
The claim now follows from the monotonicity of $b$.
\end{proof}
\begin{remark}
Note that in Theorem \ref{thm:concavity MP}, although the solution $u$ is required to be twice differential in $\Omega$ and $C^1$ up to the boundary, no regularity hypotheses are needed for the functions $f$ and $b$.
\end{remark}

\begin{remark}
In the quasi-linear setting, Theorem \ref{thm:concavity MP} recovers the original result of Korevaar \cite{Korevaar83}. An important refinement of Korevaar's result was obtained by Kennington \cite{Kennington85} (see also Kawohl \cite[Theorem 3.13]{Kawohl85}). A fully nonlinear version of Kennington's refinement can also be obtained by adapting Kawohl's argument (say) to the proof of Theorem \ref{thm:concavity MP}.
\end{remark}

\begin{remark}\label{rem:applications}
Theorem \ref{thm:concavity MP} can be applied almost immediately (cf. \cite{Korevaar83}) to nonlinear capillary problem
\begin{equation}\label{eq:capillary}
\begin{split}
f(Du,D^2u)={}&\kappa u+M\;\;\text{in}\;\;\Omega\\
\nu|_{\graph u}={}&\nu|_{\Omega} \;\;\text{on}\;\;\pd\Omega\,,
\end{split}
\end{equation}
where $\kappa>0$, 
and
\[
f(Du,D^2u)=F(A[u])
\]
is a non-decreasing, concave function of the second fundamental form $A[u]$ of $\graph u$. Indeed, although Theorem \ref{thm:concavity MP} does not directly apply (since $Du$ degenerates at the boundary of $\Omega$), it does apply to the restriction of the problem to domains $\Omega'\Subset \Omega$ which are sufficiently close to $\Omega$. So the conclusion holds in all $\Omega'\Subset \Omega$ and we conclude that $u$ is concave in $\Omega$.

In some cases, a perturbation argument (cf. \cite[Lemma 1.5]{Korevaar83}) can be used to weaken Condition (iii) to weak monotonicity, in which case the theorem applies also to certain nonlinear Weingarten problems ($\kappa=0$ in \eqref{eq:capillary}) and yields concavity properties of solutions to certain nonlinear eigenvalue problems. We will not explore such applications here, since they have been developed elsewhere (see \cite{AlvarezLasryLions,Juutinen,Kawohl85,Kennington85,Korevaar83,Sakaguchi}).
\end{remark}

In Section \ref{sec:nonlinear flows}, we apply this simple and elegant idea to certain degenerate fully nonlinear equations (namely, level set flows of convex hypersurfaces). 

Let us begin our investigation within the simpler context of \emph{mean curvature flow}, where our main result follows more or less as in Theorem \ref{thm:concavity MP}. (A more subtle argument will be required when we consider more general flows.)

\section{Mean curvature flow}

Let $\{\M^n_t\}_{t\in [t_0,T)}$ be a family of smooth, strictly convex boundaries $\M^n_t=\partial\Omega_t$ moving with normal velocity $-H\nu$, where $\nu(x,t)$ is the outward pointing unit normal to $\M^n_t$ at $x$ and $H=\dvg \nu$ is the corresponding mean curvature. Recall that the \emph{arrival time} $u:\cup_{t\in[t_0,T)}\M^n_t\to\R$ of the family $\{\M^n_t\}_{t\in [t_0,T)}$ is defined by
\[
u(p)=t\;\;\iff\;\; p\in \M^n_t\,.
\]
Note that $u$ is well-defined since the hypersurfaces move monotonically.

Let $X:M^n\times[t_0,T)\to\R^{n+1}$ be a smooth family of parametrizations $X(\cdot,t)$ of $\M^n_t$. Then
\begin{equation}\label{eq:smooth arrival time identity}
u(X(x,t))=t\,.
\end{equation}
Fix a point $q=X(x,t)$ in $\M^n_t$ and local orthonormal coordinates $\{x^i\}_{i=1}^n$ for $M^n$ about $x$ (with respect to the induced metric at time $t$). Choose the basis $\{e_i\}_{i=1}^{n+1}$ for $\R^{n+1}$ so that $e_{n+1}=\nu(x,t)$ and $e_i=\partial_iX(x,t)$ for each $i=1,\dots,n$. Differentiating \eqref{eq:smooth arrival time identity} yields the identities
\begin{equation}\label{eq:smooth arrival time identity2}
Du\cdot \partial_iX=0\;\;\text{and}\;\; -H Du\cdot\nu=1
\end{equation}
and hence
\begin{equation}\label{eq:grad u MCF}
Du=-\frac{\nu}{H}\,.
\end{equation}
Since $H=\dvg\nu$, we deduce that $u$ satisfies the \emph{level set (mean curvature) flow}
\begin{equation}\label{eq:level set MCF}
-\vert Du\vert\dvg\left(\frac{Du}{\vert Du\vert}\right)=1\,.
\end{equation}

Moreover, differentiating \eqref{eq:smooth arrival time identity2} at the point $(x,t)$, we obtain
\begin{equation}\label{eq:Hessian u MCF}
D^2u=\left(\begin{matrix} -A/H & \nabla H/H^2\\ \nabla H/H^2 & -\partial_t H/H^3\end{matrix}\right)\,.
\end{equation}
It follows that $w:=\sqrt{2(u-t_0)}$ satisfies
\begin{equation}\label{eq:Hessian w MCF}
D^2w=w^{-1}\left(\begin{matrix} -A/H & \nabla H/H^2\\ \nabla H/H^2 &-(\partial_t H+H/w^2)/H^3\end{matrix}\right)\,.
\end{equation}
This is equivalent to the bilinear form studied by Hamilton in his derivation of the differential Harnack inequality \cite{HamiltonHarnack} (and later by Chow--Chu \cite{ChowChu3}, Kotschwar \cite{Kotschwar}, and Helmensdorfer--Topping \cite{HelmensdorferTopping}, who formulated ``space-time" approaches to differential Harnack inequalities). 

Recall that the differential Harnack inequality asserts that
\begin{equation}\label{eq:Harnack MCF}
\partial_tH+2\nabla_VH+A(V,V)+\frac{H}{2(t-t_0)}\geq 0\;\;\text{for all}\;\; V\in T\M_t,\; t>t_0.
\end{equation}
It is easy to see that local concavity of $w$ is equivalent to \eqref{eq:Harnack MCF}: Fix $p\in \M^n_t$ and any $V\in T_p\R^{n+1}$. Then, either  $V$ is tangent to $\M^n_t$, in which case
\[
wD^2w(V,V)=-\frac{A(V,V)}{H}\,,
\]
or $V=\lambda(V^\top-H\nu)$ for some $\lambda\in \R$ and $V^\top\in T_p\M^n_t$, in which case
\begin{equation}\label{eq:MCF Harnack from concavity}
wD^2w(V,V)=-\frac{\lambda^2}{H}\left(\!\partial_tH+\frac{H}{2(t-t_0)}+2\nabla_{V^\top}H+A(V^{\!\top},V^{\!\top})\!\right).
\end{equation}

Since the Harnack inequality is saturated by self-similarly expanding solutions, so is local concavity of the square root of the arrival time. In fact, this is readily deduced directly: if $\M^n_t=\sqrt{t}\,\M^n_1$, for $t>0$, defines a self-similarly expanding solution, then $w=u^{\frac{1}{2}}$ is homogeneous of degree 1 since $\sqrt{t/s}\,X\in \M^n_t$ if and only if $X\in \M^n_s$. But then $D^2w$ is degenerate in radial directions.

For \emph{ancient} solutions $\{\M^n_t\}_{t\in (-\infty,T)}$, the Harnack inequality becomes
\begin{equation}\label{eq:Harnack MCF ancient}
\partial_tH+2\nabla_VH+A(V,V)\geq 0\;\;\text{for all}\;\; V\in T\M_t,\; t>-\infty\,,
\end{equation}
which, by the same argument, is seen to be equivalent to local concavity of $u$ itself.


\begin{theorem}\label{thm:concavity of the arrival time}
Let $\Omega\subset \R^{n+1}$ be a bounded, convex, open set with smooth boundary. Given $u_0\in\R$, suppose that 
$u\in C^1(\overline \Omega)$ has a single critical point, $p\in \Omega$, is twice differentiable in $\Omega\setminus\{p\}$, and satisfies
\begin{equation}\label{eq:level set MCF Dirichlet}
\begin{cases}
-\displaystyle \vert Du\vert\dvg\left(\frac{Du}{\vert Du\vert}\right)=1\;\;\text{in}\;\;\Omega\setminus\{p\}\\
u\equiv u_0\;\;\text{on}\;\;\partial\Omega\,.
\end{cases}
\end{equation}
Then $\sqrt{2(u-u_0)}$ is concave. 

If we no longer assume that $\Omega$ is bounded, but require instead that $u_0=-\infty$ and that the level sets of $u$ are bounded and convex, then $u$ is concave.
\end{theorem}
\begin{proof}
Set $w:= \sqrt{2(u-u_0)}$. Then
\[
Dw=\frac{Du}{w}
\]
and hence
\[
-\sum_{i,j=1}^n\left(\delta_{ij}-\frac{w_iw_j}{\vert Dw\vert^2}\right)w_{ij}=-\vert Dw\vert \dvg\left(\frac{Dw}{\vert Dw\vert}\right)=w^{-1}
\]
in $\Omega\setminus\{p\}$. Observe that the tangent hyperplanes to the graph of $w$ are vertical at the boundary. Indeed, the normal to the graph of $w$ is given by
\[
\mathrm{N}=\frac{(-Dw,1)}{\sqrt{1+\vert Dw\vert^2}}=\frac{(\nu,Hw)}{\sqrt{H^2w^2+1}}\,.
\]
The concavity maximum principle now implies that $w$ is concave. We proceed as in the proof of Theorem \ref{thm:concavity MP}: Let $(r_0,x_0,y_0)$ attain the minimum of the concavity function $Z$ (defined in \eqref{eq:concavity function}). Since $\graph w$ lies below its boundary tangent hyperplanes (see Remark \ref{rem:applications}), we may assume that $(r_0,x_0,y_0)$ is an interior point. So we obtain the gradient identities \eqref{eq:gradZ x}-\eqref{eq:gradZ y} and hence $p_0:=Du(x_0)=Du(y_0)$. If $p_0$ is not zero, the argument given in \cite[Theorem 3.13]{Kawohl85} implies that $Z(r_0,x_0,y_0)\ge 0$. On the other hand, if $p_0=0$, then $x_0=y_0$ (since, by hypothesis, $u$ has but one critical point) and hence $Z(r_0,x_0,y_0)=0$.

To prove the second claim, fix any point $p\in \Omega$ and any $t<u(p)$. Then $p\in \Omega_{t}:=\{q\in\Omega:u(q)>t\}$. The hypotheses on $u$ imply that $\Omega_t$ is bounded and hence, by the first part of the theorem, the function $w:\Omega_t\to\R$ given by $w(q)=(2(u(q)-t))^{\frac{1}{2}}$ is concave. Thus,
\bann
D^2u(p)={}&w^{-1}(p)D^2w(p)+\frac{Du(p)\otimes Du(p)}{w^{2}(p)}\le\frac{Du(p)\otimes Du(p)}{2(u(p)-t)}\,.
\eann
Taking $t\to-\infty$ yields the claim.
\end{proof}

Note that, for an initial hypersurface which bounds a bounded convex body, the corresponding solution to mean curvature flow remains smooth until it contracts to a single point, $p$. It follows that the arrival time is smooth away from its only critical point, $p$, and $C^1$ at $p$. In fact, Huisken \cite{Hu84} proved that the solution becomes `asymptotically round' near $p$, which actually implies that the arrival time is of class\footnote{Colding and Minicozzi \cite{ColdingMinicozziArrivalTime} proved that the arrival time of a general compact, \emph{mean} convex mean curvature flow is twice differentiable. But this result requires the full force of the structure theory for singularities in mean curvature flow. We only require here that the hypersurfaces shrink to a (not necessarily round) point.} $C^2$ \cite{Hu93}. In any case, Theorem \ref{thm:concavity of the arrival time} provides a rather simple proof of Hamilton's differential Harnack inequality.


\begin{corollary}\label{cor:Hamilton Harnack}
Let $\{\M^n_t\}_{t\in [t_0,T)}$ be a smooth family of boundaries $\M^n_t=\partial\Omega_t$ of bounded convex bodies $\Omega_t$ evolving by mean curvature. Suppose that the boundaries $\M^n_t$ contract to a point at time $T$. Then the square root $w:=(2(u-t_0))^{\frac{1}{2}}$ of the arrival time $u:\Omega_{t_0}\to\R$ is concave. Equivalently,
\[
\partial_tH+2\nabla_VH+A(V,V)+\frac{H}{2(t-t_0)}\geq 0\;\;\text{for all}\;\; V\in T\M_t\,,\;\;t\in(t_0,T)\,.
\]
If the solution is ancient, then $u$ is concave. Equivalently,
\[
\partial_tH+2\nabla_VH+A(V,V)\geq 0\;\;\text{for all}\;\; V\in T\M_t\,,\;\;t\in(t_0,T)\,.
\]
\end{corollary}
\begin{proof}
By \eqref{eq:grad u MCF}, we find that $u$ has a single critical point and is differentiable everywhere. It follows that the arrival time $u$ of the family satisfies the hypotheses of Theorem \ref{thm:concavity of the arrival time}, and we conclude that its square root $w:=(2(u-t_0))^{\frac{1}{2}}$ is concave. The differential Harnack inequality then follows from \eqref{eq:Hessian w MCF} as in \eqref{eq:MCF Harnack from concavity}. The remaining claim is proved similarly.
\end{proof}

\begin{remark}
Note that, since the level-set flow equation is not defined when $Du=0$, a separate argument in Theorem \ref{thm:concavity of the arrival time} was necessary at such points. 

After we completed this work, we learned that Trudinger had essentially pointed out the proof of Theorem \ref{thm:concavity of the arrival time} in the concluding remarks to \cite{Trudinger90}, and that Evans and Spruck \cite[Theorem 7.6]{ES91} had proved a stronger version of Theorem \ref{thm:concavity of the arrival time} by applying the concavity maximum principle to approximating solutions to the (non-singular) $\varepsilon$-regularized level-set flow and taking a limit as $\varepsilon\to 0$. Notably, both of these works preceded Hamilton's paper \cite{HamiltonHarnack}.

Xu-Jia Wang \cite[Lemma 4.1]{Wa11} observed that the \emph{logarithm} of $u-t_0$ is concave (in general), and used this to deduce that $u$ is concave for an ancient solution. His argument seems to implicitly make use of the assumptions in Theorem \ref{thm:concavity of the arrival time} and was one of the motivations for this work.
\end{remark}

\section{Flows by nonlinear functions of curvature}\label{sec:nonlinear flows}

We now consider a much larger class of evolutions. Let $\{\M^n_t\}_{t\in [t_0,T)}$ be a family of smooth, convex boundaries $\M^n_t=\partial\Omega_t$ moving with normal velocity $-F\nu$, where $\nu(x,t)$ is the outward pointing unit normal to $\M^n_t$ at $x$. We consider speeds $F(\cdot,t):\M^n_t\to\R$ given by
\begin{equation}\label{eq:speed definition}
F(x,t)=f^\alpha\big(\nu(x,t),[A_{(x,t)}]\big)
\end{equation}
for some $\alpha>0$, where $A_{(x,t)}$ is the second fundamental form of $\M^n_t$ at $x$ and $[A_{(x,t)}]$ its component matrix with respect to an orthonormal frame for $T_x\M^n_t$, and $f:S^n\times \Gamma_+^{n\times n}\to\R$ is a smooth, positive function which is $\mathrm{SO}(n)$-invariant\footnote{I.e. invariant under conjugation of its second factor by special orthogonal matrices.} and monotone non-decreasing in its second entry, where $\Gamma^{n\times n}_+$ is the cone of positive definite, symmetric $n\times n$ matrices. 

Since $f$ is positive, the hypersurfaces move monotonically inwards, so the arrival time $u:\cup_{t\in[t_0,T)}\M^n_t\to\R$, which we recall is given by
\[
u(p)=t\;\;\iff\;\; p\in \M^n_t\,,
\]
is well-defined. If the boundaries contract to a point, then the arrival time is well-defined on all of $\Omega_{t_0}$ and of class $C^1(\overline\Omega)$. If $F$ is isotropic and the boundaries contract smoothly to a `round' point, then the arrival time is of class $C^2(\overline\Omega)$. 
Indeed, the same calculations as in the preceding section reveal that
\begin{equation}\label{eq:grad u FNL}
Du=-\frac{\nu}{F}
\end{equation}
and
\begin{equation}\label{eq:Hessian u FNL}
D^2u=\left(\begin{matrix} -A/F & \nabla F/F^2\\ \nabla F/F^2 & -\partial_t F/F^3\end{matrix}\right)\,.
\end{equation}
Since, in the isotropic case,
\[
\pd_tF=\dot F(\cd^2F+FA^2)\,,
\]
where $\dot F:=Df|_{[A]}$, the claims follow similarly as in \cite{Hu93}. Moreover, $u$ satisfies the \emph{level set flow}
\[
\vert Du\vert f^\alpha\left(-\frac{Du}{\vert Du\vert},-D\frac{Du}{\vert Du\vert}\right)=1\,.
\]



Set
\[
w:=((1+\alpha)(u-t_0))^{\frac{1}{1+\alpha}}\,.
\]
Then, away from the final point,
\[
Dw=w^{-\alpha}Du\,,
\]
\begin{align}
w^{\alpha}D^2w={}&D^2u-\alpha\frac{Du\otimes Du}{w^{1+\alpha}}\nonumber\\
={}&\left(\begin{matrix} -A/F& \nabla F/F^2\\ \nabla F/F^2 &-(\partial_tF+\alpha F/w^{1+\alpha})/F^3\end{matrix}\right)\label{eq:Hessian w FNL}
\end{align}
and
\[
\vert Dw\vert f^\alpha\left(-\frac{Dw}{\vert Dw\vert},-\frac{1}{\vert Dw\vert}\left[I-\frac{Dw\otimes Dw}{\vert Dw\vert^2}\right]\cdot D^2w\right)=w^{-\alpha}\,.
\]


As in \cite{An07}, let us call a function $f:S^n\times\Gamma_+^{n\times n}\to\R$ \emph{inverse-concave} if the dual function $f_\ast:S^n\times\Gamma_+^{n\times n}\to\R$ defined by
\[
f_\ast^{-1}(p,r)=f(p,r^{-1})
\]
is concave.

\begin{theorem}\label{thm:arrival time for nonlinear flows}
Let $\Omega\subset \R^{n+1}$ be a bounded, convex, open set with smooth boundary. Given $u_0\in\R$ and $\alpha>0$, suppose that $u\in C^1(\overline \Omega)$ has a single critical point, $p\in \Omega$, is smooth in $\Omega\setminus\{p\}$, and satisfies
\[
\begin{cases}
\vert Du\vert f^\alpha\left(-\frac{Du}{\vert Du\vert},-\frac{1}{\vert Du\vert}\left[I-\frac{Du\otimes Du}{\vert Du\vert^2}\right]\cdot D^2u\right)=1&\text{in}\;\;\Omega\setminus\{p\}\\
u=u_0&\text{on}\;\;\pd\Omega\,,
\end{cases}
\]
where $f:S^n\times\Gamma_+^{n\times n}\to\R$ is monotone non-decreasing and inverse-concave. 
Then $w:=((1+\alpha)(u-u_0))^{\frac{1}{1+\alpha}}$ is concave.

If we no longer assume that $\Omega$ is bounded, but require instead that $u_0=-\infty$ and that the level sets of $u$ are bounded, then $u$ is concave. 
\end{theorem}
\begin{proof}
Consider the concavity function $Z:[0,1]\times\Omega\times\Omega\to\R$, which we recall is defined by 
\[
Z(r,x,y):= w(rx+(1-r)y)-\big(rw(x)+(1-r)w(y)\big)\,.
\]
Choose the triple $(r_0,x_0,y_0)$ so that
\[
Z(r_0,x_0,y_0)=\min_{[0,1]\times\overline\Omega\times\overline \Omega}Z(r,x,y)\,.
\]
As before, it suffices to assume that $r_0$, $x_0$ and $y_0$ are interior points. Let us abuse notation by writing $Z(x,y):=Z(r_0,x,y)$. Then $(x_0,y_0)$ is a stationary point of $Z$ and hence, setting $z_0:= r_0x_0+(1-r_0)y_0$,
\[
0=\partial_{x^i}Z(x_0,y_0)=r_0(w_i(z_0)-w_i(x_0))
\]
and
\[
0=\partial_{y^i}Z(x_0,y_0)=(1-r_0)(w_i(z_0)-w_i(y_0))\,.
\]
So
\[
Dw(z_0)=Dw(x_0)=Dw(y_0)=:p_0\,.
\]
We may also assume that $p_0\neq 0$ since 
if $p_0=0$, we would have $x_0=y_0=z_0$, and hence $Z(x_0,y_0)=0$. 

At this point, the proof differs from that of previously known results. In order to obtain the best possible result, we need to optimize the second variation inequality for $Z$ (cf. \cite{An07,AndrewsClutterbuckfundamentalgap,AndrewsClutterbuckmodulusofcontinuity}). Since $(x_0,y_0)$ is a local minimum, we obtain, for any pair of endomorphisms $a$ and $b$ of $\R^{n+1}$,
\begin{align*}
0\leq{}&\left.\frac{d^2}{ds^2}\right|_{s=0}Z(x_0+sa\cdot e_i,y_0+sb\cdot e_j)\\
={}&(a_i^p\partial_{x^p}+b_i^p\partial_{y^p})(a_j^q\partial_{x^q}+b_j^q\partial_{y^q})Z(x_0,y_0)\\
={}&(r_0a+(1-r_0)b)_i^p(r_0a+(1-r_0)b)_j^qw_{pq}(z_0)\\
{}&-r_0a_i^pa_j^qw_{pq}(x_0)-(1-r_0)b_i^pb_j^qw_{pq}(y_0)\,.
\end{align*}
The endomorphisms $a$ and $b$ will be chosen in order to optimize this inequality. Denote by
\[
\pi_0:=I-\frac{p_0\otimes p_0}{\vert p_0\vert^2}
\]
the projection onto the orthogonal complement of $p_0$. Since the equation is degenerate in the direction of $Du$, we consider only those endomorphisms of the form
\[
a=\hat a\circ \pi_0\;\;\text{and}\;\;b=\hat b\circ\pi_0\,,
\]
where $\hat a$ and $\hat b$ are endomorphisms of $\pi_0\cdot\R^{n+1}$. Then
\begin{align}\label{eq:optimize this!}
\hat c{}_i^p\hat c{}_j^q(A_{z_0})_{pq}\le{}&r_0\hat a{}_i^p\hat a{}_j^q(A_{x_0})_{pq}+(1-r_0)\hat b{}_i^p\hat b{}_j^q(A_{y_0})_{pq}\,,
\end{align}
where $\hat c:=r_0\hat a+(1-r_0)\hat b$ and
\begin{equation*}
A_x:=-\frac{1}{\vert Dw(x)\vert}\left(I-\frac{Dw(x)\otimes Dw(x)}{\vert Dw(x)\vert^2}\right)\cdot D^2w(x)\,.
\end{equation*}
%
Setting $\hat a=A_{x_0}^{-1}$ and $\hat b=A_{y_0}^{-1}$, we find
\[
r_0A_{x_0}^{-1}+(1-r_0)A_{y_0}^{-1}\le A_{z_0}^{-1}\,.
\]
The monotonicity and concavity of $f_\ast$ then yield
\begin{align*}
w(z_0)={}&\vert p_0\vert^{-\frac{1}{\alpha}}f^{-1}\left(\frac{p_0}{\vert p_0\vert},A_{z_0}\right)\\
={}&\vert p_0\vert^{-\frac{1}{\alpha}}f_\ast\left(\frac{p_0}{\vert p_0\vert},A_{z_0}^{-1}\right)\\
\geq{}&\vert p_0\vert^{-\frac{1}{\alpha}}f_\ast\left(\frac{p_0}{\vert p_0\vert},r_0A_{x_0}^{-1}+(1-r_0)A_{y_0}^{-1}\right)\\
\geq{}&r_0\vert p_0\vert^{-\frac{1}{\alpha}}f_\ast\left(\frac{p_0}{\vert p_0\vert},A_{x_0}^{-1}\right)+(1-r_0)\vert p_0\vert^{-\frac{1}{\alpha}}f_\ast\left(\frac{p_0}{\vert p_0\vert},A_{y_0}^{-1}\right)\\
={}&r_0w(x_0)+(1-r_0)w(y_0)\,.
\end{align*}
The first claim is proved. The second follows as in Theorem \ref{thm:concavity of the arrival time}.
%
%
\end{proof}

As a corollary, we obtain differential Harnack inequalities for flows by positive powers of inverse-concave speeds which contract convex hypersurfaces to points. Such inequalities were already observed by Andrews \cite[Corollary 5.11]{AndrewsHarnack} (see also Chow \cite{ChowGCF}). 

\begin{corollary}\label{cor:FNL Harnack}
Let $\{\M^n_t\}_{t\in [t_0,T)}$ be a smooth family of boundaries $\M^n_t=\partial\Omega_t$ of bounded convex bodies $\Omega_t$ moving with inward normal speed
\[
F(x,t)=f^\alpha\big(\nu(x,t),[A_{(x,t)}]\big)
\]
for some $\alpha>0$, where $f:S^n\times\Gamma^{n\times n}_+\to\R_+$ is a smooth function which is $\mathrm{SO}(n)$-invariant, monotone non-decreasing and inverse-concave in its second entry. Suppose that the hypersurfaces $\M^n_t$ contract to a point at time $T$. Then the $(1+\alpha)$-th root $w:=((1+\alpha)(u-t_0))^{\frac{1}{1+\alpha}}$ of the arrival time $u:\Omega_{t_0}\to\R$ is concave. Equivalently,
\[
\partial_tF+2\nabla_VF+A(V,V)+\!\frac{\alpha F}{(1\!+\!\alpha)(t\!-\!t_0)}\!\geq\!0\;\text{for all}\; V\in T\M_t,\; t\in(t_0,T)\,.
\]
If the solution is ancient, then $u$ is concave. Equivalently,
\[
\partial_tF+2\nabla_VF+A(V,V)\geq 0\;\;\text{for all}\;\; V\in T\M_t\,,\;\; t\in(t_0,T)\,
\]
\end{corollary}
\begin{proof}
The proof is similar to that of Corollary \ref{cor:Hamilton Harnack}.
\end{proof}

\begin{remark}
Corollary \ref{cor:FNL Harnack} assumes that the solution contracts to a single point at the singular time. This is known to be the case for solutions to isotropic flows satisfying only slightly stronger conditions than $\alpha$-inverse-concavity \cite[Theorem 5]{AMZ}. (The proof of this fact does not require differential Harnack inequalities). Moreover, examples are given in \cite{AMZ} of speeds which do not preserve convexity of the level sets $\M^n_t$, and hence cannot admit power concave arrival times. 

We do not require that the limiting shape is round. Indeed, in many situations where Harnack inequalities are known, this will not be the case \cite{AndrewsAffineNormalFlow,AndrewsCurves}.

In contrast to the known approaches to differential Harnack inequalities, Theorem \ref{thm:arrival time for nonlinear flows} does not require any regularity hypotheses for the speed.
\end{remark}

\bibliographystyle{acm}
\bibliography{../bibliography}

\end{document}